\documentclass{amsart}
\newcommand{\be}{\begin{equation}}
\newcommand{\ee}{\end{equation}}
\newcommand{\bea}{\begin{eqnarray}}
\newcommand{\eea}{\end{eqnarray}}

\newcommand{\bee}{\begin{eqnarray*}}
\newcommand{\eee}{\end{eqnarray*}}

 \newtheorem{thm}{Theorem}[section]
 
 \newtheorem{lem}[thm]{Lemma}
 \newtheorem{prop}[thm]{Proposition}
 \theoremstyle{definition}
 \newtheorem{defn}[thm]{Definition}
 \newtheorem{fac}[thm]{Fact}
 \theoremstyle{remark}

\usepackage{enumitem}
 \usepackage{color}

 \title [STUDY OF $\mathrm{b}$-GENERALIZED SKEW  DERIVATIONS ON LIE IDEAL]{A CHARACTERIZATION OF $b$-GENERALIZED SKEW  DERIVATIONS ON LIE IDEAL of A PRIME RING}
 \author[A. Pandey]{ Ashutosh Pandey }
\address{A. Pandey, School of Liberal Studies, Ambedkar University Delhi, Delhi-110006, INDIA.\\ Orcid Id 0000-0002-0002-3312}
\email{ashutoshpandey064@gmail.com}
\author[M. S. Pandey]{Mani Shankar Pandey }
\address{M. S. Pandey, Department of Sciences, Indian Institute of Information Technology Design and Manufacturing Kurnool, 518007 - India.\\ Orcid Id 0000-0003-3829-4430}
\begin{document}
\maketitle

\begin{abstract}
Let $R$ be a prime ring of characteristic different from $2$, $U$ be its Utumi quotient ring, $C$ be its extended centroid and $L$  be a non-central Lie ideal of $R$. Suppose $F$ and $G$ are two non-zero $b$-generalized skew derivations of $R$ associated with the same automorphism $\alpha$ such that $$puF(u) + F(u)uq = G(u^2),\ \text{with} \ p + q \notin C,~~ \text{for all $u \in L$. }$$
This paper aims to study the complete structure of the $\mathrm{b}$-generalized skew derivations $F$ and $G$.
\end{abstract}
\section{Introduction}
A ring $R$ is said to be prime if, for any $x, y \in R$, $xRy = 0$ implies either $x = 0$ or $y = 0$. Throughout this research article, unless specifically stated, $R$ will denote a prime ring with center $\mathcal{Z}(R)$, and $U$ will denote its Utumi quotient ring. It is worth noting that the quotient ring $U$ is also a prime ring for any prime ring $R$, and its center $C$ is termed as the extended centroid of $R$, which is a field. The definition and construction of $U$ can be found in \cite{beidar1995}. To simplify computations, we use the notation $[x, y] = xy -yx$ for all $x, y \in R$. Let $L$ and $M$ be two subrings of $R$. We define $[L, M]$ as the subgroup of $R$ generated by the set $\langle[l, m] \mid l \in L, m \in M\rangle$. A subset $L$ of $R$ is considered a Lie ideal of $R$ if it is an additive subgroup that satisfies $[L, R] \subseteq L$.

A linear mapping $\Delta:R\rightarrow R$ is said to be a derivation if $$\Delta(xy)=\Delta(x)y +x\Delta(y)$$ for all $x,y\in R$. For a fixed $p \in R$, the mapping $\Delta_p: R \rightarrow R$, defined by $\Delta_p (x ) = [p, x ]$ for all $x \in R$ is a derivation, known as inner derivation induced by an element ${p}$. A derivation that is not inner is called an outer derivation. In 1957, Posner \cite{posner1957} proved that if $\Delta$ is a non-trivial derivation of a prime ring $R$ such that $[\Delta(x), x] \in \mathcal{Z}(R)$ for all $x \in R$, then $R$ is a commutative ring. Posner's results were subsequently generalized by mathematicians in various aspects.

In 1991, M. Bre$\check{s}$ar \cite{Breaser1991} introduced a new type of derivation and termed it as generalized derivation. 
\begin{defn}
A linear mapping $\Phi: R\rightarrow R$ is said to be a generalized derivation if there exists a derivation $d$ on $R$ such that $$\Phi(xy)=\Phi(x)y+xd(y),\ \text{for all $x,y \in R$.}$$   
\end{defn}
For fixed $\alpha,\beta \in R$, the mapping $\Phi_{(\alpha,\beta)}: R \rightarrow R$ defined by $\Phi_{(\alpha,\beta)}(x) = \alpha x+ x\beta$ is a generalized derivation on $R$, and is usually termed as a generalized inner derivation on $R$.
\begin{defn}
   A linear mapping $T:{R}\longrightarrow R$ is said to be a skew derivation associated with the automorphism $\alpha\in Aut(R)$  if $$T(xy) = T(x)y+ \alpha(x)T(y), \ \text{for all $x,{y} \in {R}$.}$$
\end{defn}  
A skew derivation associated with the identity automorphism is simply a derivation. Specifically, for a fixed element $b \in {U}$ (the Utumi quotient ring), the map defined as $x \mapsto bx - \alpha(x)b$ is a well-known example of a skew derivation called an inner skew derivation. On the other hand, if a skew derivation is not of this form, it is referred to as an outer skew derivation.  
\begin{defn}
A linear mapping $\psi: R \longrightarrow R$ is called a generalized skew derivation associated with the automorphism $\alpha \in \text{Aut}(R)$ if there exists a skew derivation $\Delta$ on $R$ such that $\psi(xy) = \psi(x)y + \alpha(x)\Delta(y)$ for all $x, y \in R$.
\end{defn}  
   A skew derivation associated with the identity map is a derivation and a generalized skew derivation associated with identity automorphism is a generalized derivation.  
 In 2021, De Filippis \cite{de2020} studied the    generalized skew derivation identity $\mathcal{F}(\mathcal{G}(x))=0$ for all $x\in L$, on the Lie ideal  $L$ of a prime ring $R$, where $\mathcal{F},\ \mathcal{G}$ are generalized skew derivations on the prime ring $R$. 
 
 In a study conducted in 2018 by De Filippis and Wei  \cite{de2018}, the notion of $\mathrm{b}$-generalized skew derivation was introduced. This notion expands upon the concept of derivations and explores various types of linear mappings in the context of non-commutative algebras. 
  \begin{defn}
Let $R$ be an associative  ring,  $\mathrm{d}:R \rightarrow U$ be a linear mapping, $\alpha$ be an automorphism of $R$, and $\mathrm{b}$ be a fixed element in the Utumi quotient ring $U$. Then a linear mapping $F : R \rightarrow U$, is said to be a $\mathrm{b}$-generalized skew derivation of $R$ associated with the triplet  $(\mathrm{b},\alpha,\mathrm{d})$ if $$F(xy) = F(x)y+ \mathrm{b}\alpha(x)\mathrm{d}(y),\ \text{for all $x,y\in R$.}$$ 
\end{defn}
Moreover, the authors have proved that if  $\mathrm{b}\neq 0$, the associated additive map $\mathrm{d}$, defined above, is a skew derivation. Further, it has been shown that the additive mapping $F$ can be extended to the Utumi quotient ring $U$ and it assumes the form $F(x)= \mathrm{a}x+ \mathrm{b}\mathrm{d}(x)$, where $\mathrm{a} \in U$. The concept of  $\mathrm{b}$-generalized skew derivation with the associated triplet $(\mathrm{b}, \alpha, \mathrm{d})$ covers the concepts of skew derivation, generalized derivations, and left multipliers, etc.  For instance if we choose $\mathrm{b}= 1$, $\mathrm{b}$-generalized skew derivation becomes a skew derivation and for $\mathrm{b} = 1, \alpha=I_{R}$, $\mathrm{b}$-generalized skew derivation becomes a generalized derivation, where $I_{R}$ is the identity map on the ring $R$. Further, if we take $\mathrm{b}=0$, then the $\mathrm{b}$-generalized skew derivation becomes a left multiplier map.  Let $R$ be a prime ring and $\alpha$ be an automorphism of $R$. Then the mapping $F : R\longrightarrow U$ given by $ x \mapsto  ax + \mathrm{b}\alpha(x)c$ is a $\mathrm{b}$-generalized skew derivation of $R$ with associated triplet $(\mathrm{b},\alpha,\mathrm{d})$, where $a, b$ and $c$ are fixed elements in $R$ and $d(x) = \alpha(x)c - cx$, for all $x \in R$, such $\mathrm{b}$-generalized skew derivation is called  inner $\mathrm{b}$-generalized skew derivation.

From the preceding definition, it can be concluded that the general results regarding $\mathrm{b}$-generalized skew derivations offer valuable and impactful corollaries concerning derivations, generalized derivations, and generalized skew derivations. These findings have the potential to provide significant insights and applications in the study of these related concepts.

 It is very natural to see what happens whenever we replace the derivations  by $\mathrm{b}$-generalized skew derivations in Posner's and Bre$\check{s}$ar's results. In 2021, Filippis et. al.  \cite{V. D. Filippis 2021} attempted and succeeded partially in giving a generalization of  Bre$\check{s}$ar's result by studying the $\mathrm{b}$-generalized skew derivation identity $\mathcal{F}(u)u - u\mathcal{G}(u) = 0$ on a prime ring $R$, for all $u\in\{\phi(x)\ |\ x=(x_1,\ldots,x_n) \in \ {R}^n\}$,  where $\phi(x)$ is a multilinear polynomial over $C$ and $\mathcal{F},\ \mathcal{G}$ are $\mathrm{b}$-generalized skew derivations on the prime ring $R$. Many generalizations which are relevant to our discussion about $\mathrm{b}$-generalized skew derivations can be found in \cite{A23, de2018, de 2018, V. D. Filippis 2021}. Continuing our investigation, we focus on the $\mathrm{b}$-generalized skew derivation identity $puF(u) + F(u)uq = G(u^2)$, where $p + q \not\in C$ for all $u \in L$. We present the theorem that establishes our result:

\begin{thm} \label{AMP}
Let $R$ be a prime ring of characteristic different from $2$, $U$ be its Utumi quotient ring, $C$ be its extended centroid and $L$ be a non-central Lie ideal of $R$. Suppose $F$ and $ G$ be two non-zero $b$-generalized skew derivations of $R$ with associated terms $(b, \alpha, d)$ and $(b,\alpha, h)$ respectively, satisfying the identity
$$puF(u) + F(u)uq = G(u^2)\ \text{for some}\ p,q\in R\ \text{with}\ p + q \not\in C, \ \forall\ u \in L.$$ Then for all $x \in R$ one of the following holds:
\begin{enumerate}
    \item There exists $a \in C$, and $c \in U$ for which $F(x) = ax, G(x) = xc$ with $pa =
c - aq \in C$ and $p \in C$.
\item There exists $a \in C$, and $c, c^{\prime} \in U$ for which $F(x) = ax$, $G(x) = cx + xc^{\prime}$ with
$(pa- c)=c^{\prime}-qa\in C$ and $pa-c\in C$.
\item There exists $a \in C$, and $c \in U$ for which $F(x) = ax, G(x) = cx$ with $(p+q)a = c$
and $q \in C$.
\item There exist  $a,c^{\prime}\in C$ and $c \in U$ for which $F(x) = ax$, $G(x) = (c+c^{\prime})x$ with $aq =
c +c^{\prime}- ap \in C$.
\item There exists $a,c\in C$ and $c^{\prime}\in U$ for which $F(x)=ax$ and $G(x)=x(c+c^{\prime})$ with $p\in C$ and $pa=(c+c^{\prime})-qa\in C.$
\item There exists $a,b, u,v,t \in U$ and $\lambda, \eta \in C$ such that $F(x)=(a+bu)x$, $G(x)=cx+btxt^{-1}v$ with $t^{-1}v+\eta  q=\lambda\in C$, $(a+bu)+\eta bt=0,$ and $p(a+bu)-c= \lambda bt$.
\item If $R$ satisfies $s_4$ and we have the following conclusions:
\begin{enumerate}
    \item There exist $a,b,c,u,v \in U$ such that $F(x) = (a+bu)x$, $G(x) = cx + xbv$ with $(p+q)(a+bu)=(c+bv)$ for all $x \in R$.
   \item There exist $a,b,c,u,v \in U$ and an invertible $t\in U$ such that $F(x) = (a+bu)x$, $G(x) = cx + btxt^{-1}v$ for all $x \in R$ with $(p+q)(a+bu)=(c+bv)$. 
\end{enumerate}
\end{enumerate}
\end{thm}
The general strategy for attacking the main theorem can be used to demonstrate
a more general result for  multilinear polynomials. As a result, from a purely
theoretical standpoint, there is no difference between the Lie ideal case and the
multilinear polynomial case. Adopting the proof in the case of multilinear polynomials is to avoid more calculations.

The organization of the paper is as follows. In Section $2$, we recall some basic facts about prime rings. The next section discusses the case when the $\mathrm{b}$-generalized skew derivations $F$ and $G$ are inner. In Section $4$, we prove our main theorem by considering each case in detail. 
\section{Preliminaries and Notations}
We frequently utilize the following facts to establish our results:
\begin{fac}\cite{beidar1978ri}\label{fac2}
	Let $R$ be a prime ring and $I$ a two-sided ideal of $R$. Then $R$, $I$, and $U$ satisfy the same generalized polynomial identities with coefficients in $U$.
\end{fac}
\begin{fac} \cite{lee1992semiprime}\label{fac3}
	Let $R$ be a prime ring and $I$ a two-sided ideal of $R$. Then $R, I$ and $U$
	satisfy the same differential identities.
\end{fac}
\begin{fac}\cite{beidar1978ri}\label{fac4}
	Let $R$ be a prime ring. Then every derivation $d$ of $R$ can be uniquely
	extended to a derivation of $U$.
\end{fac}
\begin{fac}\label{fac5}
	(Chuang, \cite{chuang2005identities}) Let $R$ be a prime ring, $d$ be a nonzero skew derivation on $R$, and $I$ a nonzero ideal of $R$. If $I$ satisfies the differential identity.
	\begin{equation*}
			f\big(x_1,x_2,\ldots,x_n,d(x_{1}),d(x_{2})\ldots, d({x_n})\big)=0
		\end{equation*}
	for any $x_1,\ldots,x_n\in I$, then either
	\begin{itemize}
		\item $I$ satisfies the generalized polynomial identity
		\begin{equation*}
			f(x_1,x_2\ldots,x_n,y_1,y_2\ldots, y_n)=0
			\end{equation*} for all $y_1,\ldots,y_n \in R$.\\
		or
		\item $d$ is $U$-inner.
		\begin{equation*}
			f\big(x_1,x_2,\ldots,x_n,[p,x_1],[p,x_2]\ldots, [p,x_n]\big)=0
		\end{equation*}
	\end{itemize}
\end{fac}
\begin{fac}\cite{de2012}\label{fac9}
	Let $K$ be an infinite field and $m\geq 2$ an integer. If $P_1,\ldots,P_k$ are non-scalar matrices in $M_m(K)$ then there exists some invertible matrix $P \in M_m(K)$ such that each matrix $PP_1P^{-1},\ldots,PP_kP^{-1}$ has all non-zero entries.
	\end{fac}
\begin{fac} \cite{argacc2011actions}\label{fac 8}
Let $R$ be a non-commutative prime ring of characteristic not equal to $2$ with Utumi quotient ring $U$ and extended centroid $C$ and let $f(x_1,\ldots,x_n)$ be a multilinear polynomial over $C$ which is not central valued on $R$. Suppose that there exists $a,b,c\in U$ such that $f(r)af(r)+f(r)^2b-cf(r)^2=0$ for all $r_1,\ldots, r_n\in R.$ Then one of the following holds:
\begin{enumerate}
   \item $b,c\in C$, $c-b=a=\alpha \in C,$
\item $f(x_1,\ldots,x_n)^2$ is central valued and there exists $\alpha \in C$ such that $c-b=a=\alpha$.
\end{enumerate}
\end{fac}
\begin{fac}\cite{de2018}\label{fac 9}
If $d$ is a non-zero skew derivation on a prime ring $R$ then associated automorphism $\alpha$ is unique.
\end{fac}
\begin{fac} \cite{de2020}\label{x}
    Let $R $ be a prime ring and $\phi,\gamma$ be two automorphisms of $U$ and $d,g$ be two skew derivations on $R$ associated with the same automorphism $\phi$. If there exists a non-zero central element $\nu$ and $v\in U$ such that 
    \begin{align*}
        g(x)=(vx-\gamma(x)v)+\nu d(x),\ \text{for all}\ x\in R
     \end{align*}
Then $g(x)=\nu d(x)$ and one of the following holds:
\begin{enumerate}
    \item $\phi=\gamma$ 
    \item $v=0$.
\end{enumerate}
\end{fac}
\begin{fac}  \cite{de2020}\label{1}
 Let $R $ be a prime ring and $\phi,\gamma$ be two automorphisms of $U$ and $d,g$ be two skew derivations on $R$ associated with the same automorphism $\phi$. If there exists a non-zero central element $\nu$ and $v\in U$ such that 
    \begin{align*}
        g(x)=(vx-\gamma(x)v)+\nu d(x),\ \text{for all}\ x\in R
     \end{align*}
If $d$ is inner skew derivation then so $g$.
\end{fac}
The standard polynomial identity $s_4$ in four variables is defined as
$$s_4(x_1, x_2, x_3, x_4) = \sum\limits_{\sigma\in Sym(4)}(-1)^{\sigma}x_{\sigma(1)}x_{\sigma(2)}x_{\sigma(3)}x_{\sigma(4)},$$ where $(-1)^\sigma$ is either $+1$ or $-1$ according to $\sigma$ being an even or odd permutation in the symmetric group $Sym(4)$. Note that in this article ${R}$ always denotes a non-trivial and associative prime ring (unless otherwise stated) and we use the abbreviation GPI for generalized polynomial identity. 

\section{$F$ and $G$ are inner $b$-generalized skew derivations}
In this section, we investigate the case where $F$ and $G$ are inner $b$-generalized skew derivations of $R$ with the associated term $(b, \alpha)$. Specifically, we consider $F(x) = ax + b\alpha(x)u$ and $G(x) = cx + b\alpha(x)v$ for all $x \in R$, where $a, b, c, u, v \in U$.

To establish our main result, we present the following Lemmas.
\begin{lem}\label{prop1A}
	Let $R$ be a prime ring of $char(R) \neq 2$ and $ a_1$, $a_2$, $a_3$, $a_4$ and $a_5 \in R$ such that 
	\begin{equation}\label{eqA12}
		a_1u^2 +a_2u^2a_3 +a_4u^2a_5=0,\  \forall\ u =[r_1,r_2]\in [R,R].
	\end{equation}
 Then one of the following holds:
	\begin{enumerate}
	\item $R$ satisfies $s_4$.
	\item $a_3,a_4 \in C$ and $a_1 +a_2a_3 = -a_5a_4 \in C$.
	\item $a_3,a_5 \in C$ and $a_1+a_2a_3+a_5a_4 =0$.
	\item $a_1, a_2,a_4 \in C$ and $a_1+a_2a_3+a_5a_4 =0$.
	\item $a_2, a_5 \in C$ and $a_1 +a_4a_5 =-a_2a_3 \in C$.
	\item there exists $\lambda$, $\eta$, $\mu \in C$ s.t. $a_5 +\eta a_3=\lambda$, $a_2 -\eta a_4= \mu$ and $a_1 +\lambda a_4= -\mu a_3 \in C$
	\end{enumerate}
\end{lem}
\begin{proof}
	If $u^2$ is a centrally valued element in $R$, then $R$ satisfies $s_4$. Now, suppose that $u^2$ is not central, and let $S$ be the additive subgroup of $R$ generated by $S = \{u^2: u \in [R,R]\}$. It is evident that $S \neq 0$, and we have $a_1x + a_2xa_3 + a_4xa_5 = 0$ for all $x \in S$.

According to \cite{chuang1987additive}, either $S \subset Z(R)$, or $char(R) = 2$ and $R$ satisfies $s_4$, unless $S$ contains a non-central ideal $L^{\prime}$ of $R$. Since $u^2$ is not centrally valued in $R$, the first case is not possible. Also since
$char(R) \neq  2$, therefore $S$ contains a non-central Lie ideal $L^{\prime}$ of $R$. By \cite{bergen1981lie}, there exists a non-central
two-sided ideal $I$ of $R$ such that $[I, R] \subseteq L^{\prime}$. By the hypothesis, we have:
	\begin{equation*}
	a_1[r_1,r_2] +a_2[r_1,r_2]a_3 +a_4[r_1,r_2]a_5=0
	\end{equation*} 
 for all $r_1,r_2 \in I$. From Fact \ref{fac2}, $R, I$, and $U$ satisfy the same GPI, we have
 \begin{equation*}
	a_1[r_1,r_2] +a_2[r_1,r_2]a_3 +a_4[r_1,r_2]a_5=0
	\end{equation*} for all $r_1,r_2 \in R$. Hence from [\cite{de2021some}, Proposition 2.13], we get our conclusions.
	\end{proof}
\begin{lem}\label{lem1}
Let $R = M_m(C)$, $m \geq 2$, be the ring of all $m \times m$ matrices over an infinite field $C$ with characteristics not equal to $2$ and $a, b, c, u, v, p, q \in R$ such that 
\begin{align*}
    pXaX +pXbXu+aX^2q+bXuXq-cX^2-bX^2v = 0
\end{align*}
for all $X=[x_1,x_2] \in [R,R]$, then either $b \in C$ or $u \in C$ or $p + q \in C$.
\end{lem}
\begin{proof}
Suppose the field $C$ is infinite. From the hypothesis, $R$ satisfies the following generalized polynomial identity
\begin{equation}\label{eq1}
   pXaX +pXbXu+aX^2q+bXuXq-cX^2-bX^2v = 0 
\end{equation} 
for all $X\in [R,R]$. Assume that $p+q,b$, and $u$ are noncentral elements. Now since Equation (\ref{eq1}) is invariant under the action of any automorphism on $R$, then by Fact \ref{fac9} we can assume that all the entries of $p+q, b, u$ are nonzero. Choosing $X= e_{ij}$ in Equation (\ref{eq1}), we get
\begin{equation}\label{eq2}
   pe_{ij}ae_{ij} +pe_{ij}be_{ij}u+be_{ij}ue_{ij}q=0.
\end{equation} Now right and left multiplying Equation (\ref{eq2}) by $e_{ij}$, we get
\begin{equation*}
    (p+q)_{ji}u_{ji}b_{ji}e_{ij}=0
\end{equation*} which implies either $(p+q)_{ji}=0$ or $u_{ji}=0$ or $b_{ji}=0$. In each case, we get a contradiction, therefore  either $p+q \in C$ or $b\in C$ or $u\in C$.
\end{proof}
\begin{lem}\label{le}
Let $R = M_m(C)$, $m \geq 2$, be the ring of all $m \times m$ matrices over a  field $C$ with characteristics not equal to $2$ and $a, b, c, u, v, p, q \in R$ such that 
\begin{align*}
    pXaX +pXbXu+aX^2q+bXuXq-cX^2-bX^2v = 0
\end{align*}
for all $X=[x_1,x_2] \in [R,R]$, then either $b \in C$ or $u \in C$ or $p + q \in C$.
\end{lem}
\begin{proof}
If $C$ is an infinite field, we are done by Lemma \ref{lem1}. Otherwise, let's consider the case when the field $C$ is finite. Let $K$ be an infinite extension field of $C$, and let $\bar{R} = M_m(K) \cong R \otimes_C K$. It is important to note that a multilinear polynomial is centrally valued on $R$ if and only if it is centrally valued on $\bar{R}$.

Consider the generalized polynomial identity for $R$ given by
\begin{align}\label{eq3}
Q(r_1, r_2) = p[r_1,r_2]a[r_1,r_2] + p[r_1,r_2]b[r_1,r_2]u + a[r_1,r_2]^2q \nonumber\\
+ b[r_1,r_2]u[r_1,r_2]q - c[r_1,r_2]^2 - b[r_1,r_2]^2v.
\end{align}
This polynomial is of multidegree $(2, 2)$ with respect to the indeterminates $r_1$ and $r_2$. Hence, the complete linearization of $Q(r_1, r_2)$ gives a multilinear generalized polynomial $\Theta(r_1, r_2, x_1, x_2)$ in four indeterminates. Furthermore, we have $\Theta(r_1, r_2, r_1, r_2) = 4Q(r_1, r_2)$.

It is evident that the multilinear polynomial $\Theta(r_1, r_2, x_1, x_2)$ is a generalized polynomial identity for both $R$ and $\bar{R}$. Since the characteristic of $R$ is different from $2$ by assumption, we obtain $Q(r_1, r_2) = 0$ for all $r_1, r_2 \in \bar{R}$. Thus, the conclusion follows from Lemma \ref{lem1}.
\end{proof}
\begin{lem}\label{lem2}
Let $R$ be a prime ring of characteristic different from 2 with Utumi
quotient ring $U$ and extended centroid $C$. Suppose that for some $a, b, c, u, v, p, q
\in R$\begin{equation*}
   pXaX +pXbXu+aX^2q+bXuXq-cX^2-bX^2v = 0 
\end{equation*} 
for all $X=[x_1,x_2] \in [R,R]$. Then either $b \in C$ or $u$ or $p + q \in C$.
\end{lem}
\begin{proof}
\textbf{Case 1:}
Suppose that none of $b$, $u$, or $p+q$ is central. By hypothesis, we have
\begin{align}h(x_1,x_2) = p[x_1,x_2]a[x_1,x_2] + p[x_1,x_2]b[x_1,x_2]u \nonumber\\
+a[x_1,x_2]^2q+ b[x_1,x_2]u[x_1,x_2]q-c[x_1,x_2]^2 - b[x_1,x_2]^2v
\end{align}
for all $x_1,x_2 \in R$. Let $T = U \star_{C} C\{x_1, x_2\}$, which is the free product of $U$ and $C\{x_1, x_2\}$, the free $C$-algebra in non-commuting indeterminates $x_1, x_2$. Since $R$ and $U$ satisfy the same generalized polynomial identities (GPI) (see Fact \ref{fac2} and \ref{fac3}), $U$ satisfies $h(x_1, x_2) = 0$ in $T$.

Suppose that $h(x_1, x_2)$ is a trivial GPI for $R$. Then, $h(x_1, x_2)$ is a zero element in $T = U *_{C} C\{x_1 , x_2\}$. However, since $b$, $u$, and $p+q$ are not central, the terms $b[x_1,x_2]u[x_1,x_2]q$ and $p[x_1,x_2]b[x_1,x_2]u$ appear non-trivially in $h(x_1,x_2)$, which leads to a contradiction. Therefore, we conclude that at least one of $b$, $u$, or $p+q$ belongs to $C$.

\textbf{Case 2:}
Next, suppose that $h(x_1,x_2)$ is a non-trivial GPI for $U$. In the case when $C$ is infinite, we have $h(x_1, x_2) = 0$ for all $x_1, x_2 \in U \otimes_{C} \bar{C}$, where $\bar{C}$ is the algebraic closure of $C$. Since both $U$ and $U \otimes_{C} \bar{C}$ are prime and centrally closed (see Theorem 2.5 and 3.5 of \cite{erickson1975pr}), we may replace $R$ by $U$ or $U \otimes_{C} \bar{C}$ depending on whether $C$ is finite or infinite. Then $R$ is centrally closed over $C$ and $h(x_1,x_2) = 0$ for all $x_1, x_2 \in R$.

By Martindale’s theorem \cite{martindale1969pr}, $R$ is a primitive ring with nonzero socle, $soc(R)$ and with $C$ as its associated division ring. Then, by Jacobson’s theorem (see p.75 of \cite{N. Jacobson1956}), $R$ is isomorphic to a dense ring of linear transformations of a vector space $V$ over $C$. Assume first that $V$ is finite-dimensional over $C$, i.e., $dim_{C} V = m$. By the density of $R$, we have $R \cong M_m(C)$. Since $R$ is non-commutative, we have $m \geq 2$. In this case, we obtain our conclusions by Lemma \ref{lem1}.

Next, we suppose that $V$ is infinite-dimensional over $C$. For any $e^2 = e \in \text{soc}(R)$, we have $eRe \cong M_t(C)$ with $t = \dim_{C} Ve$. We shall prove this case by contradiction. Since none of $b, u$, and $p + q$ are central, there exist $h_1, h_2, h_3 \in \text{soc}(R)$ such that $[b, h_1] \neq 0$, $[u, h_2] \neq 0$, and $[p + q, h_3] \neq 0$. By Litoff’s Theorem \cite{c}, there exists an idempotent $e \in \text{soc}(R)$ such that $bh_1, h_1b, uh_2, h_2u, (p + q)h_3, h_3(p + q), h_1, h_2, h_3 \in eRe$. Since $R$ satisfies the generalized identity, we rewrite this as:
\begin{align}
    e\big\{p[ex_1e , ex_2e]a[ex_1e , ex_2e] + p[ex1e , ex_2e]
b[ex_1e , ex_2e]u\nonumber \\ + a[ex_1e , ex_2e]^2q
 + b[ex_1e , ex_2e]u[ex_1e , ex_2e]q\nonumber\\
 - c[ex_1e , ex_2e]^2 - b[ex_1e , ex_2e]^2v \big\}e=0
\end{align}
for all $x_1,x_2 \in R$. Then the subring $eRe$ satisfies
\begin{align}
    epe[x_1,x_2]eae[x_1,x_2] + epe[x_1,x_2]ebe[x_1,x_2]eue\nonumber \\
+eae[x_1,x_2]^2eqe+ ebe[x_1,x_2]eue[x_1,x_2]eqe\nonumber \\
-ece[x_1,x_2]^2-ebe[x_1,x_2]^2eve=0.
\end{align}
for all $x_1,x_2 \in R$. Then by the above finite-dimensional case, either $ebe$ or $eue$ or $e(p + q)e$ is a central
element of $eRe$. Thus either $bh_1 = (ebe)h_1 = h_1ebe = h_1b$ or $uh_2 = (eue)h_2 =
h_2(eue) = h_2u$ or $(p + q)h_3 = e(p + q)eh_3 = h_3(e(p + q)e) = h_3(p + q)$, a contradiction.  Thus we have either $ b \in C$ or $u \in C$ or $p + q \in C$.
\end{proof}
\begin{prop}\label{prop1}
Let $R$ be a prime ring of characteristic different from 2, $U$ be its Utumi ring of quotients  with extended centroid $C$, and $L$ be a Lie ideal of $R$. Let $F, G$ be two $b$-generalized skew inner derivations of $R$ with associated pair $(b, \alpha)$. Suppose there exist elements $p, q \in R$ such that
\begin{equation*}
    puF(u)+F(u)uq= G(u^2),\ \text{with}\  p+q \notin C,\ \forall\ u \in L.
\end{equation*} Then for all $x \in R$ one of the following holds:
\begin{enumerate}
    \item There exists $a \in C$, and $c \in U$ for which $F(x) = ax, G(x) = xc$ with $pa =
c - aq \in C$ and $p \in C$.
\item There exists $a \in C$, and $c, c^{\prime} \in U$ for which $F(x) = ax$, $G(x) = cx + xc^{\prime}$ with
$(p + q)a = c + c^{\prime}$.
\item There exists $a \in C$, and $c \in U$ for which $F(x) = ax, G(x) = cx$ with $(p+q)a = c$
and $q \in C$.
\item There exist  $a,c^{\prime}  \in C$ and $c \in U$ for which $F(x) = ax$, $G(x) = (c+c^{\prime})x$ with $aq =
c +c^{\prime}- ap \in C$.
\item There exists $a,c\in C$ and $c^{\prime}\in U$ for which $F(x)=ax$ and $G(x)=x(c+c^{\prime})$ with $p\in C$ and $pa=(c+c^{\prime})-qa\in C.$
\item There exists $a,b, u,v,t \in U$ and $\lambda, \eta \in C$ such that $F(x)=(a+bu)x$, $G(x)=cx+btxt^{-1}v$ with $t^{-1}v+\eta  q=\lambda\in C$, $(a+bu)+\eta bt=0,$ and $p(a+bu)-c= \lambda bt$.
\item If $R$ satisfies $s_4$ and we have the following conclusions:
\begin{enumerate}
    \item There exist $a,b,c,u,v \in U$ such that $F(x) = (a+bu)x$, $G(x) = cx + xbv$ with $(p+q)(a+bu)=(c+bv)$ for all $x \in R$.
   \item There exist $a,b,c,u,v \in U$ and an invertible $t\in U$ such that $F(x) = (a+bu)x$, $G(x) = cx + btxt^{-1}v$ for all $x \in R$ with $(p+q)(a+bu)=(c+bv)$. 
\end{enumerate}
\end{enumerate}
\end{prop}
\begin{proof}
From the hypothesis, $R$ satisfies:
\begin{equation}\label{eq8}
    pXaX+pXb\alpha(X)u+aX^2q+b\alpha(X)uXq-cX^2-b\alpha(X^2)v=0
\end{equation} for all $X=[x_1,x_2] \in [R,R]$.

    \textbf{Case 1 :} Suppose the associated automorphism $\alpha$ with the skew derivations $d$ and $g$ is inner, then there exists an invertible element $t \in R$ such that $\alpha(x)= txt^{-1}$  for all $X \in [R,R]$ and Equation (\ref{eq8}) becomes:
\begin{equation}\label{eq9}
  pXaX+pXbtXt^{-1}u+aX^2q+btXt^{-1}uXq-cX^2-btX^2t^{-1}v=0  
\end{equation}
for all $X \in [R,R]$. Then from Lemma \ref{lem2}, either $bt \in C$ or $t^{-1}u \in C$.

\textbf{Subcase $(a)$ :} If $bt \in C$ then Equation (\ref{eq9}) reduces to:
\begin{equation}\label{eq10}
     pXaX+pX^2bu+aX^2q+XbuXq-cX^2-X^2bv=0  
\end{equation} for all $X \in [R,R]$. Again by previous arguments, one of the following holds:
\begin{enumerate}
    \item $p,pa, bu \in C$.
    \item $a, bu \in C$.
    \item $a,q,buq \in C$.
\end{enumerate}
Now we discuss each of the above cases in detail.
\begin{enumerate}
\item Suppose $pa,p,bu \in C$ then Equation (\ref{eq10}) reduces to:
\begin{equation}\label{eq111}
   (pa-c)X^2+aX^2q+X^2(pbu+buq-bv)=0 
\end{equation} for all $X=[x_1,x_2] \in [R,R]$. Then from Lemma \ref{prop1A} one of the following holds:
\begin{itemize}
    \item $X^2$ is central valued, (i.e. $R$ satisfies $s_4$), then from Equation (\ref{eq111}), we get $(pa-c+aq+pbu+buq-bv)X^2=0$, which implies $(p+q)(a+bu)=(c+bv).$ Thus in this situation we get $F(x) = (a+bu)x$, $G(x) = cx + xbv$   for all $x \in R$ with $(p+q)(a+bu)=(c+bv)$. This is our Conclusion   $7(a)$.
   \item $q\in C,$ this gives that $p+q\in C$, a contradiction.
    \item $q,(pbu+buq-bv)\in C$ this gives that $p+q\in C$, a contradiction
    \item $(pa-c),a\in C$,  this gives that $c\in C$. Thus from Equation (\ref{eq111}), we get $(a+bu)p=(c+bv)-(a+bu)q$. Hence in this case we get $F(x) = (a+bu)x$, $G(x) = x(c+bv)$ for all $x \in R$ with  $(a+bu)p=(c+bv)-(a+bu)q\in C$, which is our Conclusion $(1)$.
    \item $q,(pbu+buq-bv)\in C$ this gives that $p+q\in C$, a contradiction
    \item there exist $\eta,\ \lambda,\ \mu\in C$ such that $(pbu+buq-bv)+\eta q=\lambda$, $a-\eta=\mu$ and $(pa-c)+\lambda=-\mu q\in C$, this gives that $q\in C$ which implies $p+q\in C$, a contradiction.\\ If $\mu=0$ then $pa-c, a \in C$ then by previous arguments we get our Conclusion (1).
    
\end{itemize}
\item Suppose $a, bu \in C$ then Equation (\ref{eq10}) reduces to:
\begin{equation}\label{eq11a}
   (p(a+bu)-c)X^2+aX^2q+X^2(buq-bv)=0
\end{equation} for all $X=[x_1,x_2] \in [R,R]$. Then from Lemma \ref{prop1A} one of the following holds:
\begin{itemize}
    \item $X^2$ is central valued ( i.e. $R$ satisfies $s_4$), then from Equation (\ref{eq11a}), we get  $(p(a+bu)-c)+((a+bu)q-bv)=0$ which implies $ (p+q)(a+bu)=(c+bv)$. Thus in this case we have$F(x) = (a+bu)x$, $G(x) = cx + xbv$ for all $x \in R$ with $ (p+q)(a+bu)=(c+bv)$. This is our Conclusion   $7(a)$.
    \item $q,(buq-bv)\in C$, which gives that $bv\in C$. Thus from the Equation (\ref{eq11a}), we get that
    $(pa+pbu-c+aq+buq-bv)X^2=0$, which implies $(a+bu)q=(c+bv)-(a+bu)p\in C$. Hence in this case we get $F(x) = (a+bu)x$, $G(x) = (c+bv)x$ for all $x \in R$ with $(a+bu)q=(c+bv)-(a+bu)p\in C$, which is our Conclusion $(4)$.
    \item $q,(buq-bv)\in C$, this gives that $bv\in C$. In this situation, we get our conclusion  by previous arguments.
    \item $(p(a+bu)-c) \in C$ then from Equation (\ref{eq11a}) we get that $X^2 (pa+pbu-c+aq+buq-bv)=0$ which implies $(p+q)(a+bu)=(c+bv)$. Thus in this case we get $F(x) = (a+bu)x$, $G(x) = cx + xbv$ for all $x \in R$ with $(p+q)(a+bu)=(c+bv)$, which is our conclusion $(2)$.
    \item $(buq-bv),q\in C$,  this gives that $bv\in C$. Then by the previous arguments, we get our Conclusion $(4)$.
    \item there exist $\eta,\lambda,\mu\in C$ such that $(buq-bv)+\eta q=\lambda$, $a-\eta=\mu$ and $ (p(a+bu)-c)+\lambda=-\mu q\in C$, this gives that $bv\in C$. Thus again by the previous arguments we get our  Conclusion $(4)$.\\ If $\mu=0$ then $p(a+bu)-c\in C$, thus by previous arguments we get our Conclusion (4). 
\end{itemize}
\item Suppose $a,q,buq \in C$, then Equation (\ref{eq10}) reduces to:
\begin{equation}\label{eqa1}
   (ap+aq+buq-c)X^2+pX^2bu-X^2bv=0
\end{equation} for all $X \in [R,R]$. Then from Lemma \ref{prop1A} one of the following holds:
\begin{itemize}
    \item $X^2$ is central valued (i.e. $R$ satisfies $s_4$), then from Equation (\ref{eqa1}) we get $(ap+aq+buq-c+pbu-bv)=0$  which implies $ (p+q)(a+bu)=(c+bv)$ and our functions takes the form $F(x) = (a+bu)x$, $G(x) = cx + xbv$ for all $x \in R$ with $(p+q)(a+bu)=(c+bv)$. This is our Conclusion   $7(a)$.
    \item $bu,bv\in C$ and  our functions takes the form $F(x) = (a+bu)x$, $G(x) = (c+bv)x$ with   $(a+bu)q=(c+bv)-(a+bu)p\in C$ for all $x \in R$  which is our Conclusion $(4)$.
    \item $bu,bv\in C$. In this situation, we get our
conclusion from previous arguments.
    \item $(ap+aq+buq-c),p\in C$ this gives $p+q\in C$, a contradiction.
    \item $p,bv\in C$ this gives $p+q\in C$, a contradiction.    
    \item  there exists $\eta,\lambda,\mu\in C$ such that $-bv+\eta bu=\lambda,$ $p-\eta=\mu$ and $ (ap+aq+buq-c)+\lambda =-\mu bu\in C$, this gives that $p=\eta +\mu\in C$ and hence $p+q\in C$, a contradiction.
 \end{itemize}
\end{enumerate}
\textbf{Subcase $(b)$:} If $t^{-1}u \in C$, then Equation (\ref{eq9}) reduces to:
\begin{equation}\label{eq91}
  pXaX+pXbuX+aX^2q+buX^2q-cX^2-btX^2t^{-1}v=0  
\end{equation}
for all $X \in [R,R]$.  Again by previous arguments, one of the following holds:
\begin{enumerate}
\item $p,p(a+bu)\in C$
\item $a+bu\in C$
 \end{enumerate}
 \begin{enumerate}
     \item Now if $p,p(a+bu)\in C$  Equation (\ref{eq91}) transformed in to
\begin{equation}\label{eq101}
(p(a+bu)- c)X^2+(a+bu)X^2q-btX^2t^{-1}v=0  
\end{equation} 
by Lemma  \ref{prop1A} one of the following holds:
\begin{itemize}
\item $X^2$ is central valued (i.e. $R$ satisfies $s_4$), then from Equation (\ref{eq101})  $-c-bv+(p+q)(a+bu)=0$ and our functions takes the form $F(x) = (a+bu)x$, $G(x) = cx + btxt^{-1}v$ for all $x \in R$ with $(p+q)(a+bu)=(c+bv)$. This is our Conclusion   $7(b)$.
    \item $q,bt\in C$ which implies $p+q\in C$, a contradiction.
    \item $t^{-1}v,q\in C$  which implies $p+q\in C$, a contradiction..
    \item $p(a+bu)-c,(a+bu),bt\in C$ with $p(a+bu)-c+(a+bu)q=bv$ which implies $c\in C$ and $p(a+bu)=-q(a+bu)+(c+bv)\in C$ and our functions takes the form $F(x) = (a+bu)x$, $G(x) = x(c+bv)$  which is our Conclusion $(5)$.
    \item $(a+bu),t^{-1}v\in C$ with $p(a+bu)-c+bv=-(a+bu)q\in C$. Since $F$ and $G$ are nonzero, $a+bu\neq 0$. Therefore,  $p+q\in C$ is a contradiction.    
    \item there exists $\eta,\lambda,\mu\in C$ such that $t^{-1}v+\eta  q=\lambda\in C$, $(a+bu)-\eta bt=\mu\in C$ and $p(a+bu)-c+\lambda bt=-\mu q\in C.$  If $\mu\neq 0$, then $q\in C$ and thus $p+q\in C$, a contradiction.
    \\Again if $\mu=0$ then $t^{-1}v+\eta  q=\lambda\in C$, $(a+bu)+\eta bt=0,$ and $p(a+bu)-c= \lambda bt$. Thus in this situation, we get our conclusion(6).
\end{itemize}
\item 
Now if $(a+bu)\in C$ and Equation (\ref{eq91}) transformed in to
\begin{equation}\label{eq11}
  (p(a+bu)-c)X^2-btX^2t^{-1}v+X^2(a+bu)q=0  
\end{equation}
for all $X \in [R,R]$. By Lemma  (\ref{prop1A}) one of the following holds:
\begin{itemize}
    \item $X^2$ is central valued, i.e. $R$ satisfies $s_4$ and $(p(a+bu)-c)+bv+(a+bu)q=0$ and our functions takes the form $F(x) = (a+bu)x$, $G(x) = cx + btxt^{-1}v$ for all $x \in R$ with $(p+q)(a+bu)=(c+bv)$. This is our Conclusion   $7(b)$.
    \item $t^{-1}v\in C$ and $(p(a+bu)-c)-bv=-(a+bu)q\in C$. In this situation functions $F$ and $G$ take the form, $F(x) = (a+bu)x$, $G(x) = (c+bv)x$ for all $x \in R$ with $(q(a+bu)=(c+bv)-p(a+bu)$. This is  our Conclusion $(5)$.
    \item $t^{-1}v,(a+bu)q\in C$ and $(p(a+bu)+(a+bu)q-c-bv=0$. In this situation, we get our conclusion from previous arguments.
    \item $(p(a+bu)-c),bt\in C$ and $(p(a+bu)+(a+bu)q-c-bv=0$. The functions $F$ and $G$ take the form, $F(x) = (a+bu)x$, $G(x) = cx + xbv$ for all $x \in R$ with $(p+q)(a+bu)=(c+bv)$. This is  our Conclusion $(2)$.
    \item $bt,(a+bu)q,t^{-1}v\in C$ and $(p(a+bu)-c+(a+bu)q=bv\in C$. Thus functions $F$ and $G$ take the form, $F(x) = (a+bu)x$, $G(x) = (c + bv)x$ with $(q(a+bu)=(c+bv)-(a+bu)p$ for all $x \in R$. This is our Conclusion   $(5)$.
    \item there exists $\eta,\lambda,\mu\in C$ such that $(a+bu)q+\eta t^{-1}v=\lambda$, $-bt-\eta=\mu$, $ (p(a+bu)-c)+\lambda=-\mu t^{-1}v\in C$. This gives that $q,t^{-1}v\in C$. The functions $F$ and $G$ take the form, $F(x) = (a+bu)x$, $G(x) = (c+bv)x$ with $(p+q)(a+bu)=(c+bv)$, $q\in C$ for all $x \in R$. This is our Conclusion   $(3)$. \\
    Now if $\mu=0$, then $bt, p(a+bu)-c) \in C.$ Then by previous arguments we get our Conclusion $(2)$.
\end{itemize}
 \end{enumerate}
\textbf{Case $2$ :}
 If $\alpha$ is outer then from \cite{chuang1988gpis}, Equation (\ref{eq8}) reduces to
\begin{equation}\label{eq8a}
    p[x_1,x_2]a[x_1,x_2]+p[x_1,x_2]b[y_1,y_2]u+a[x_1,x_2]^2q+b[y_1,y_2]u[x_1,x_2]q\end{equation}\begin{equation*}-c[x_1,x_2]^2-b[y_1,y_2]^2v=0
\end{equation*} for all $x_1,x_2,y_1,y_2\in R.$ In particular $R$ satisfies the following
\begin{align}\label{eq8b}
p[x_1,x_2]a[x_1,x_2]+p[x_1,x_2]b[x_1,x_2]u+a[x_1,x_2]^2q\nonumber \\
+b[x_1,x_2]u[x_1,x_2]q-c[x_1,x_2]^2-b[x_1,x_2]^2v=0, \ \forall \ x_1,x_2\in R.
\end{align}  Thus by similar arguments used in Lemma \ref{lem1}, we can show that either $b \in C$ or $u \in C$. If $b \in C$ then from Equation(\ref{eq8b}), we get
\begin{equation}\label{eq8c}
    p[x_1,x_2]a[x_1,x_2]+p[x_1,x_2]^2bu+a[x_1,x_2]^2q+[x_1,x_2]bu[x_1,x_2]q\end{equation}\begin{equation*}-c[x_1,x_2]^2-[x_1,x_2]^2bv=0, \ \forall \ x_1,x_2\in R.
\end{equation*} 
Now, Equation (\ref{eq8b}) is similar to Equation (\ref{eq10}), thus we get the required results by similar arguments used in Subcase (a) of Case (1). Again if $u \in C$ then 
from Equation (\ref{eq8b}), we get
\begin{align}\label{eq8d}
p[x_1,x_2]a[x_1,x_2]+p[x_1,x_2]bu[x_1,x_2]+a[x_1,x_2]^2q\nonumber \\
+bu[x_1,x_2]^2q-c[x_1,x_2]^2-b[x_1,x_2]^2v=0,\ \forall\ x_1,x_2\in R.
\end{align}
Now, Equation (\ref{eq8d}) is similar to Equation (\ref{eq91}), thus the conclusion follows from the Subcase (b) of Case 1.
\end{proof}
Now we give the proof of the Theorem \ref{AMP}.
\section{Proof of Theorem \ref{AMP}}
This final section of our paper is dedicated to the proof of the main theorem (Theorem \ref{AMP}). Throughout the proof, we assume that $R$ is not contained in $M_2(C)$. According to \cite{lee1999}, we may assume that there exist $a, c \in U$ and skew derivations $d$ and $g$ associated with the automorphism $\alpha$, such that $F(x) = ax + bd(x)$ and $G(x)= cx+bg(x)$ for all $x \in R$. Since $L$ is non-central and $char(R) \neq 2$, there exists a nonzero ideal $J$ of $R$ such that $0 \neq [J, R] \subseteq L$ (see \cite{herstein1969t}, p.45; \cite{OM1989}, Lemma 2 and Proposition 1; \cite{lanski1972li}, Theorem 4). Therefore, we have:
 $$pXF(X)+F(X)Xq=G(X^2),\ \forall\  X \in [J, J].$$
 Since $R$ and $J$ satisfy the same generalized differential identities, we also have $pXF(X)+F(X)Xq=G(X^2)$, for all $X \in [R, R]$. Then by the hypothesis, $R$ satisfies
 \begin{align} \label{m1}
     p[x,y]a[x,y]+p[x,y]bd([x,y])+a[x,y]^2q+bd([x,y])[x,y]q\nonumber \\
     -c[x,y]^2     -bg([x,y])^2)=0
 \end{align}
 for all $x,y\in R$.
  \begin{center}
     \textbf{$d$ is skew inner and $g$ is skew outer derivation}
 \end{center}
 Since $d$ is  an inner skew derivation of $R$, there exists $b^{\prime}\in U$ such that $d(x)=b^{\prime}x-\alpha(x)b^{\prime}$ for all $x\in {R}$ and Equation (\ref{m1}) reduced to
 \begin{align} \label{m11}
     p[x,y]a[x,y]+p[x,y]b(b^{\prime}([x,y]-\alpha([x,y])b^{\prime})+a[x,y]^2q+b(b^{\prime}([x,y]\nonumber \\
     -\alpha([x,y])b^{\prime})[x,y]q-c[x,y]^2     -bg([x,y])^2)=0
 \end{align}
 for all $x,y\in R$. Apply the definition of $g$ we have 
\begin{align} \label{m2}
     p[x,y]a[x,y]+p[x,y]b(b^{\prime}([x,y]-\alpha([x,y])b^{\prime})
     +a[x,y]^2q \nonumber \\
     +b(b^{\prime}([x,y]     -\alpha([x,y])b^{\prime})[x,y]q-c[x,y]^2    -b\bigg\{\big(g(x)y+\alpha(x)g(y) -g(y)x \nonumber \\-\alpha(y)g(x)\big)[x,y]
     +\alpha([x,y])\big(g(x)y+\alpha(x)g(y)-g(y)x-\alpha(y)g(x)\big)\bigg\}     =0
 \end{align}
for all $x,y\in R$. Since $g$ is an outer derivation of $R$, by Chung’s theorem (see Fact \ref{fac5}) 
in Equation (\ref{m2}), we obtain
 \begin{align} \label{m21}
     p[x,y]a[x,y]+p[x,y]b(b^{\prime}([x,y]-\alpha([x,y])b^{\prime})
     +a[x,y]^2q+b(b^{\prime}([x,y]    \nonumber \\ -\alpha([x,y])b^{\prime})[x,y]q
     -c[x,y]^2  \nonumber \\ -b\bigg\{\big(wy+\alpha(x)z-zx-\alpha(y)w\big)[x,y]     +\alpha([x,y])\big(wy+\alpha(x)z-zx-\alpha(y)w\big)\bigg\}  \nonumber \\   =0
 \end{align}
  for all $x,y,w,z\in R$. In particular taking $z=0$ in Equation (\ref{m21}), we get 
 \begin{align} \label{m3}
b\bigg\{\big(\alpha(x)z-zx\big)[x,y]     +\alpha([x,y])\big(\alpha(x)z-zx\big)\bigg\}     =0
 \end{align}
  for all $x,y,z\in R$. Now, if the automorphism $\alpha$ is not inner,  then from \cite{chuang1988gpis}, Equation (\ref{m3}) reduces to 
\begin{align*} 
b\bigg\{\big(vz-zx\big)[x,y]     +[v,u])\big(vz-zx\big) \bigg\}      =0
 \end{align*}
 for all $x,y,v,z\in R$. In particular 
 \begin{align*} 
2b[x,y]^2     =0, \implies b=0
 \end{align*}
 for all $x,y\in R$, but this makes both $F$ and $G$ inner $b$-generalized skew derivations, contradicting our assumption. Further, if the automorphism is inner,  then there exists $t\in U$ such that $\alpha(x)=txt^{-1}$ and  Equation (\ref{m3}) reduces to 
\begin{align} \label{m4}
b\bigg\{\big(txt^{-1}z-zx\big)[x,y]     +t[x,y]t^{-1}\big(txt^{-1}z-zx\big)\bigg\}     =0
 \end{align}
for all  $x,y,z\in R$. In particular, choosing $z=ty$
 \begin{align*} 
2bt[x,y]^2     =0 \implies bt=0 \Rightarrow b=0
 \end{align*}
 but this again leads to a contradiction.
 \begin{center}
     \textbf{$d$ is skew outer  and $g$ is skew inner derivation}
     \end{center}
       In this case there exists $c\in U$ such that $g(x)=cx-\alpha(x)c$ for all $x\in {R}$ and Equation (\ref{m1}) reduced to
       \begin{align} \label{m5}
     p[x,y]a[x,y]+p[x,y]b(d(x)y+\alpha(x)d(y)-d(y)x-\alpha(y)d(x))\nonumber \\
     +a[x,y]^2q+b(d(x)y     +\alpha(x)d(y)-d(y)x-\alpha(y)d(x))[x,y]q     \nonumber \\
     -c[x,y]^2     -b(c[x,y]^2-\alpha([x,y]^2)c)
     =0
 \end{align}
 for all $x,y\in R$. Since $d$ is an outer derivation on $R$,  from Fact \ref{fac5},  Equation (\ref{m5}) reduces to 
 \begin{align} \label{m51}
     p[x,y]a[x,y]+p[x,y]b(wy+\alpha(x)z-zx-\alpha(y)w) -c[x,y]^2     -b(c[x,y]^2\nonumber \\
     +a[x,y]^2q+b(wy    +\alpha(x)z-zx-\alpha(y)w)[x,y]q     
    -\alpha([x,y]^2)c)     =0
 \end{align}
 for all $x,y,z,w\in R$. In particular choosing $z=0$ in Equation (\ref{m51}), $R$ satisfies the blended component
\begin{align} \label{m6}
     p[x,y]b(\alpha(x)z-zx)     +b(\alpha(x)z-zx)[x,y]q        =0
 \end{align}
 for all  $x,y,z\in R$. Now if the automorphism $\alpha$ is not inner,  then from \cite{chuang1988gpis}, Equation (\ref{m6}) reduces to 
 \begin{align*} 
     p[x,y]b(wz-zx)     +b(wz-zx)[x,y]q        =0
 \end{align*} 
 for all $x,y,w,z\in R$. In particular 
  \begin{align*} 
     p[x,y]b[x,y]    +b[x,y]^2q        =0
 \end{align*} 
 for all $x,y\in R$. Then by Lemma \ref{prop1A}, \ref{lem1}, \ref{le}, \ref{lem2}, we have  $(p+q)\in C$, which is a contradiction. Further if the automorphism $\alpha$ is  inner  then there exists $t\in U$ such that $\alpha(x)=txt^{-1}$ and  Equation (\ref{m6}) reduces to 
  \begin{align} \label{m61}
     p[x,y]b(txt^{-1}z-zx)     +b(txt^{-1}z-zx)[x,y]q        =0
 \end{align}
for all  $x,y,z\in R$. It follows from Posner's theorem \cite{posner1960pri} that there exists a suitable field $K$ and a positive integer $l$ such that $R$ and $M_l(K)$ satisfy the same polynomial identities. For instance, $M_l(K)$ also satisfies Equation (\ref{m61}). Since $R$ does not satisfy $s_4$, we must have $l \geq 3$. Choosing $x=e_{jj}$, $z=te_{ij}$ and $y=e_{ji}$ in Equation (\ref{m6}) we have 
 \begin{align*} 
     p[e_{jj},e_{ji}]b(te_{jj}e_{ij}-te_{ij}e_{jj})     +b(te_{jj}e_{ij}\nonumber \\
     -te_{ij}e_{jj})[e_{jj},e_{ji}]q =pe_{ji}bte_{ij}+bte_{ij}e_{ji}q       =0
 \end{align*}
 which on further solving gives
 \begin{align}  \label{m7}
   pe_{ji}bte_{ij}+bte_{ii}q       =0.
 \end{align}
 Now multiply by $e_{kk}\ (j\neq k)$ from both side of Equation (\ref{m7}) we get
 \begin{align}  \label{m71}
e_{kk}bte_{ii}qe_{kk}=(bt)_{ki}(q)_{ik}  e_{kk}     =0.
 \end{align}
 This gives one of $bt$ or $q$ is central. Suppose $bt\in C$ then from Equation (\ref{m7}), we have
 \begin{align}  \label{m8}
  bt pe_{ji}e_{ij}+bte_{ii}q =bt(pe_{jj}+e_{ii}q)      =0.
 \end{align}
 Since $bt$ is central, multiply from left side by $(bt)^{-1}$ in above Equation (\ref{m8}), we get 
 \begin{align}  \label{m9}
 pe_{jj}+e_{ii}q     =0.
 \end{align}
 Multiply by $e_{kk}$ $(k\neq i)$ in above Equation (\ref{m9}), we get that
 \begin{align}  
 e_{kk}pe_{jj}=p_{kj} e_{kj}  =0
 \end{align}
 but this gives that $p$ is central. Again by right multiplying by $e_{kk}$ in Equation (\ref{m9}) we can show that $q\in C$ and this gives a contradiction to our initial assumption that $p+q$ is non-central. Now let $q\in C$. Then from Equation (\ref{m7}), we get
 \begin{align}  \label{m10}
   pe_{ji}bte_{ij}+qbte_{ii} =0.
 \end{align}
 Now multiply both side by $e_{jj} $ in above Equation (\ref{m10}), we get that
  \begin{align}  \label{m101}
   e_{jj}pe_{ji}bte_{ij}e_{jj} \implies p_{ji}(bt)_{ji}=0.
 \end{align}
 This gives either $p$ is central or $bt$ is central, but in both cases, we get a contradiction.
 Now suppose both $d$ and $g$ are skew outer derivations then we have the following cases.
 \begin{center}
     \textbf{$d$   and $g$ are $C$-modulo independent}
     \end{center}
     In this case after applying the definition of $d$ and $g$, $R$ satisfies the following equation 
 \begin{align} \label{k2}
     p[x,y]a[x,y]+p[x,y]b(d(x)y+\alpha(x)d(y)-d(y)x-\alpha(y)d(x))
     +a[x,y]^2q     +b(d(x)y\nonumber \\
     +\alpha(x)d(y)-d(y)x-\alpha(y)d(x))[x,y]q     -c[x,y]^2        -b\bigg\{\big(g(x)y+\alpha(x)g(y)      -g(y)x\nonumber \\
     -\alpha(y)g(x)\big)[x,y]+\alpha([x,y])\big(g(x)y+\alpha(x)g(y)-g(y)x-\alpha(y)g(x)\big)\bigg\}=0
 \end{align}
 for all $x,y\in R$. Then by Chungs’s theorem (see Fact \ref{fac5}), Equation (\ref{k2}) reduces to
 \begin{align} \label{k3}
     p[x,y]a[x,y]+p[x,y]b(wy+\alpha(x)z-zx-\alpha(y)w)
     +a[x,y]^2q -c[x,y]^2 \nonumber \\
     +b(wy+\alpha(x)z-zx-\alpha(y)w)[x,y]q
        -b\bigg(\big(vy+\alpha(x)s-sx-\alpha(y)v\big)[x,y]\nonumber \\
     +\alpha([x,y])\big(vy+\alpha(x)s-sx-\alpha(y)v\big)\bigg)=0
 \end{align}
 for all $x,y,v,s,w,z\in R$. Choose $s=0$ in  Equation (\ref{k3}), 
 \begin{align} \label{k5}
     b\bigg\{(\alpha(x)s-sx)[x,y]     +\alpha([x,y])(\alpha(x)s-sx)\bigg\}=     0
 \end{align}
  for all $s,x,y\in R$. Now, Equation (\ref{k5}) is similar to Equation (\ref{m3}). Therefore, by similar arguments used in Case $1$, we get a contradiction.
  \begin{center}
     \textbf{$d$   and $g$ are $C$-modulo dependent}
     \end{center}
In this scenario, we consider the case ${F}({x}) = a{x} +bd({x})$ and ${G}({x}) = c{x} + bg ({x})$ for all ${x} \in {R}$  and suitable $a, c \in U$, and $d$ and $g$ are non-zero skew derivations of $R$ associated with the automorphism $\alpha$. Furthermore, assume that $d$ and $g$ are linearly $C$-dependent modulo inner skew derivations. Then there exist $\eta,\tau\in C$, $v\in U$ and an automorphism $\phi$ of $R$ such that
$\eta d(x)+\tau g(x)=vx-\phi(x)v,\ \text{ for all}\ x\in R.$\\
\textbf{Case 1:} $\eta\neq 0$ and $\tau\neq 0$. Then we have \vspace{-.1cm}
\begin{align*}
    d(x)= a_1g(x)+(a_2x-\phi(x)a_2), \ \text{for all}\ x\in R, \text{where}\ a_1=-\eta^{-1}\tau,\ a_2=\eta^{-1}v
\end{align*}
    It is worth noting that if $d$ is an inner skew derivation, then, according to Fact \ref{1}, $g$ also becomes an inner skew derivation. In this case, the conclusion can be derived as a consequence of Proposition \ref{prop1}. Hence, in the subsequent analysis, we assume that $d$ is a non-zero outer skew derivation. Consequently, based on Fact \ref{x}, we can conclude that either $\phi = \alpha$ or $a_2 = 0$. In summary, we reach one of the following conclusions: 
\begin{enumerate}
    \item $d(x)= a_1g(x)+(a_2x-\alpha(x)a_2)$
    \item $d(x)= a_1g(x).$
\end{enumerate}
We will now demonstrate that each of the conditions mentioned above leads to a contradiction. For the sake of brevity, we will focus on Case (1), as it can be shown that Case (2) follows from Case (1). Therefore, let $d(x) = a_1g(x) + (a_2x - \alpha(x)a_2)$ for all $x \in R$.

Based on Fact \ref{fac2} and Fact \ref{fac3}, it is well-known that both $J$ and $R$ satisfy the same generalized identities with a single skew derivation. Consequently, Equation (\ref{m1}) is also satisfied by $R$. By substituting the value of $d(x)$ in Equation (\ref{m1}), we can conclude that $R$ satisfies
 \begin{align}\label{x1} 
     p[x,y]a[x,y]+p[x,y]b\big(a_1g([x,y])+(a_2[x,y]-\alpha([x,y]a_2)\big) -c[x,y]^2 \nonumber \\
     +a[x,y]^2q     +b\big(a_1g([x,y])+(a_2[x,y]-\alpha([x,y]a_2)\big)[x,y]q        -bg([x,y])^2)=0
 \end{align}
 for all $x,y\in R$. Now, applying the definition of $g$ in Equation (\ref{x1}), we have
\begin{align}\label{x2} 
     p[x,y]a[x,y]+a[x,y]^2q+b\big(a_1    \big(g(x)y+\alpha(x)g(y)
-g(y)x-\alpha(y)g(x)\big)\nonumber \\
     +p[x,y]b\bigg\{a_1\big(g(x)y+\alpha(x)g(y)-g(y)x-\alpha(y)g(x)\big)+(a_2[x,y]-\alpha([x,y]a_2)\bigg\}\nonumber \\
-b\bigg\{\big(g(x)y+\alpha(x)g(y)-g(y)x-\alpha(y)g(x)\big)[x,y]
     +\alpha([x,y])\big(g(x)y+\alpha(x)g(y)\nonumber \\
     -g(y)x-\alpha(y)g(x)\big)\bigg\}+(a_2[x,y]-\alpha([x,y]a_2)\big)[x,y]q     -c[x,y]^2  =0
 \end{align}
 for all $x,y\in R$. Then using  Fact \ref{fac5} in Equation (\ref{x2}), we obtain
\begin{align}\label{x3} 
     p[x,y]a[x,y]+a[x,y]^2q-c[x,y]^2 \nonumber \\
     +p[x,y]b\bigg\{a_1\big(wy+\alpha(x)v-vx-\alpha(y)w\big)+(a_2[x,y]-\alpha([x,y]a_2)\bigg\}\nonumber \\
     -b\bigg\{\big(wy+\alpha(x)v-vx-\alpha(y)w\big)[x,y]
     +\alpha([x,y])\big(wy+\alpha(x)v-vx-\alpha(y)w\big)\bigg\}\nonumber \\
+b\bigg\{a_1\big(wy+\alpha(x)v
-vx-\alpha(y)w\big)+(a_2[x,y]-\alpha([x,y]a_2)\bigg\}[x,y]q      =0
 \end{align}
for all $x,y,w,v\in R$. Now choosing $w=0$ in Equation (\ref{x3}), we have 
\begin{align}\label{x4} 
     p[x,y]b\big(a_1(wy-\alpha(y)w)\big)+b\big(a_1(wy-\alpha(y)w)\big)[x,y]q   \nonumber \\
-b\big\{(wy-\alpha(y)w)[x,y]     +[\alpha(x),\alpha(y)])(wy-\alpha(y)w)\big\}=0
 \end{align}
 for all $x,y,w,v\in R$. Again from \cite{chuang2005identities}
\begin{align}\label{x5} 
     p[x,y]b\big(a_1(wy-su)\big)+b\big(a_1(wy-su)\big)[x,y]q   \nonumber \\
-b\big\{(wy-su)[x,y]     +[w,s](wy-su)\big\}=0
 \end{align}
 for all $x,y,w,u\in R$. In particular choosing $x=0$ in Equation (\ref{x5}), $R$ satisfies the blended component
\begin{align} \label{00}
     b[w,s](wy-su)=0
 \end{align}
for all $w,s,y,u\in R.$ In particular choosing $y=s$ and $u=w$ in Equation (\ref{00}), we have
\begin{align} 
     b[w,s]^2=0
 \end{align}
for all $w,s\in R,$ which gives that $b=0$, a contradiction.

\textbf{Case 2:} $\eta=0$ and $\tau\neq 0$. Then we have $$g(x)= a_2x-\phi(x)a_2, \ \text{for all}\ x\in R, \text{where}\ a_2=\tau^{-1}v.$$
for all $x,y\in R$. In this case, we can assume that the skew derivation $d$ is not inner. If it were inner, the conclusion would follow from Proposition \ref{prop1}. Additionally, since the automorphism associated with a skew derivation is unique, in this scenario, we have $\phi = \alpha$. Therefore, $R$ satisfies
\begin{align}\label{y3} 
 p[x,y]a[x,y]+p[x,y]bd([x,y])+a[x,y]^2q+bd([x,y])[x,y]q\nonumber \\
     -c[x,y]^2     -b\big(a_2[x,y])^2-\alpha([x,y]^2a_2)\big)=0
 \end{align}
for all $x,y\in R$. Apply the definition of $d$ in Equation (\ref{y3}) we have
\begin{align}\label{y4} 
 p[x,y]a[x,y]+p[x,y]b(d(x)y+\alpha(x)d(y)-d(y)x-\alpha(y)d(x))+a[x,y]^2q-c[x,y]^2   \nonumber \\
+b(d(x)y+\alpha(x)d(y)-d(y)x-\alpha(y)d(x))[x,y]q
  -b\big(a_2[x,y])^2-\alpha([x,y]^2a_2)\big)=0
 \end{align}
for all $x,y \in R$. Then by use of Fact \ref{fac5} in Equation (\ref{y4}), we obtain
\begin{align}\label{y5} 
 p[x,y]a[x,y]+p[x,y]b(wy+\alpha(x)z-zx-\alpha(y)w)+a[x,y]^2q-c[x,y]^2  \nonumber \\
+b(wy+\alpha(x)z-zx-\alpha(y)w)[x,y]q   -b\big(a_2[x,y])^2-\alpha([x,y]^2a_2)\big)=0
 \end{align}
for all $x,y,w,z \in R$. In particular, choosing $z=0$ in Equation (\ref{y5}), $R$ satisfies the blended component
\begin{align}\label{y6} 
 p[x,y]b(\alpha(x)z-zx)+b(\alpha(x)z-zx)[x,y]q=0
 \end{align}
$x,y,z\in R$. The above Equation (\ref{y6}) is similar to Equation (\ref{m6}), therefore this case also leads to a contradiction.

\textbf{Case 3:} $\eta\neq 0$ and $\tau= 0$. Then we have $$d(x)= a_1x-\phi(x)a_1, \ \text{for all}\ x\in R, \text{where}\ a_1=\eta^{-1}v.$$
In a similar manner to Case $2$, here we are assuming that $g$ is not inner and $\alpha=\phi$. Hence $R$ satisfies
\begin{align} \label{z1}
     p[x,y]a[x,y]+p[x,y]b(a_1[x,y]-\alpha([x,y]a_1))
     +a[x,y]^2q+b(a_1[x,y]\nonumber \\
           -b\bigg\{\big(g(x)y+\alpha(x)g(y)-g(y)x-\alpha(y)g(x)\big)[x,y]     +\alpha([x,y])\big(g(x)y+\alpha(x)g(y)\nonumber \\
           -g(y)x           -\alpha(y)g(x)\big)\bigg\}  -\alpha([x,y]a_1))[x,y]q-c[x,y]^2  =0
 \end{align}
 for all $x,y\in R$. Then from Fact \ref{fac5}, we have
\begin{align} \label{z2}
     p[x,y]a[x,y]+p[x,y]b(a_1[x,y]-\alpha([x,y]a_1))     
     -\alpha([x,y]a_1))[x,y]q-c[x,y]^2      \nonumber \\    -b\bigg\{\big(wy+\alpha(x)z-zx     -\alpha(y)w\big)[x,y]     +\alpha([x,y])\big(wy+\alpha(x)z-zx-\alpha(y)w\big)\bigg\}  \nonumber \\    
     +a[x,y]^2q+b(a_1[x,y]=0
 \end{align}
 for all $x,y,w,z\in R$. Choosing $z=0$ in Equation (\ref{z2}), we have
\begin{align}\label{z3} 
 b\big\{\alpha(x)z-zx)[x,y]+\alpha([x,y])(\alpha(x)z-zx)\big\}=0
 \end{align}
for all  $x,y,z\in R$. The above Equation (\ref{z3}) is similar to Equation (\ref{m3}), therefore this case also leads to a contradiction.
 \section*{Acknowledgment}
The authors would like to express their sincere thanks to the reviewers and referees for their valuable comments and suggestions.

\end{document}